\documentclass[10pt]{article}

\usepackage{amsmath}    % need for subequations
\usepackage{graphicx}   % need for figures
\usepackage{verbatim}   % useful for program listings
\usepackage{color}      % use if color is used in text

\usepackage{amssymb}
\usepackage{amsthm}

\theoremstyle{definition}
\newtheorem{Def}{Definition}[section]

\newtheorem{Lem}[Def]{Lemma}
\newtheorem{Thm}[Def]{Theorem}
\newtheorem{Cor}[Def]{Corollary}
\newtheorem{Rem}[Def]{Remark}

\newcommand{\p}{\mathbb{P}}
\newcommand{\e}{\mathbb{E}}
\newcommand{\real}{\mathbb{R}}
\newcommand{\n}{\mathbb{N}}

\newcommand{\1}{{\bf 1}}

\newcommand{\diffns}{\mathrm{d}}

\usepackage{slashbox}
\allowdisplaybreaks

\begin{document}

\title{Strong rate of convergence for the Euler-Maruyama approximation of SDEs with H\"older continuous drift coefficient}
\author{
Olivier Menoukeu Pamen\footnote{Institute for Financial and Actuarial Mathematics, Department of Mathematical Sciences, University of Liverpool, L69 7ZL, United Kingdom, Email:Menoukeu@liverpool.ac.uk}
~and~
Dai Taguchi\footnote{Ritsumeikan University, 1-1-1 Nojihigashi, Kusatsu, Shiga, 525-8577, Japan, Email: dai.taguchi.dai@gmail.com }
}
\date{}
\maketitle

\begin{abstract}
In this paper, we consider a numerical approximation of the stochastic differential equation (SDE)
\begin{align*}
	X_t=x_0+ \int_0^t b(s, X_s)\diffns s + L_t,~x_0 \in \real^d,~t \in [0,T],
\end{align*}
where the drift coefficient $b:[0,T] \times \real^d \to \real^d$ is H\"older continuous in both time and space variables and the noise $L=(L_t)_{0 \leq t \leq T}$ is a $d$-dimensional L\'evy process.
%In general, convergence of numerical methods are studied under the classical assumption that the coefficient of the SDE is globally Lipschitz continuous.
%However, this assumption is not always satisfied for SDEs used in practice, this makes the global Lipschitz based idea not immediately applicable.
%In this paper, we assume that the drift $b$ is H\"older continuous in both time and space variables.
We provide the rate of convergence for the Euler-Maruyama approximation when $L$ is a Wiener process or a truncated symmetric $\alpha$-stable process with $\alpha \in (1,2)$.
Our technique is based on the regularity of the solution to the associated Kolmogorov equation.
%Traditionally, the Kolmogorov equation is used to study the weak convergence rate for the Euler-Maruyama approximation (see for example \cite{ProTal}  and references therein) which is for example very important in financial applications.
%We consider multi-dimensional stochastic differential equations (SDE) with H\"older continuous drift coefficient and constant diffusion coefficient.
%We provide the rate of strong convergence for the Euler-Maruyama approximation when driving process of SDE is a Wiener process, a symmetric $\alpha$-stable process and a truncated symmetric $\alpha$-stable process with $\alpha \in (1,2)$.
\end{abstract}

\section{\large Introduction}
Let $X=(X_t)_{0\leq t \leq T}$ be the unique strong solution to the following $d$-dimensional SDE
\begin{align}\label{SDE_1}
X_t=x_0+ \int_0^t b(s, X_s)\diffns s + L_t,~x_0 \in \real^d,~t \in [0,T],
\end{align}
where $L=(L_t)_{0\leq t \leq T}$ is a $d$-dimensional L\'evy process on a  complete probability space $(\Omega, \mathcal{F},\p)$.
The drift coefficient $b:[0,T] \times \real^d \to \real^d$ is assumed to be $\eta$-H\"older continuous in time with $\eta \in [1/2,1]$ and bounded and $\beta$-H\"older continuous in space, i.e.,
\begin{align*}
\|b\|_{C_b^{\beta}([0,T])} := \sup_{t \in [0,T], x \in \real^d}|b(t,x)|+\sup_{t \in [0,T], x \neq y}\frac{|b(t,x)-b(t,y)|}{|x-y|^{\beta}}<\infty.
\end{align*}
Here $\beta$ satisfies some conditions (see Theorem \ref{main_1.5} and \ref{main_2}).

Consider the Euler-Maruyama approximation of SDE \eqref{SDE_1} given by
\begin{align}\label{EM_1}
   X_t^{(n)}
&= x_0 +\int_0^tb\left(\eta_n(s), X_{\eta _n(s)}^{(n)}\right)\diffns s +L_t,~t \in [0,T],
\end{align}
where $\eta _n(s) = kT/n=:t_k^{(n)}$ if $ s \in \left[kT/n, (k+1)T/n \right)$.
%In application it is very important to determine the rate of this approximation.
It is well-known (see for example \cite{KP}) that if $L$ is a Wiener process and the coefficient $b$ is Lipschitz continuous in space and $1/2$-H\"older continuous in time then the Euler-Maruyama scheme has strong rate of convergence $1/2$, i.e. for any $p>0$, there exists $C_p>0$ such that
$$\e \left[ \sup_{0 \leq t \leq T}\left|X_t - X^{(n)}_t\right|^p\right] \leq \frac{C_p}{ n^{p/2}}.$$

In general, convergence of numerical methods are studied under the above assumption that the coefficient of the SDE is globally Lipschitz continuous. However, this assumption is not always satisfied for SDEs used in practice (for example mathematical finance, optimal control problem and filtering), this makes the global Lipschitz based idea not immediately applicable. This lead to the study of Euler-Maruyama approximation for SDEs with irregular coefficients.

When $L$ is a Wiener process and $b$ is a continuous function satisfying a linear growth condition, Kaneko and Nakao \cite{KaNa} proved that the unique strong solution of the SDE \eqref{SDE_1} if it exists can be constructed as the limit of the Euler-Maruyama approximation.
However, they do not study the rate of convergence for the Euler-Maruyama scheme.
Hashimoto \cite{Hashimoto} considers Euler-Maruyama approximation for solutions of one dimensional SDE driven by a symmetric $\alpha$-stable process. Under the Komatsu condition (a H\"older type condition), the author shows strong convergence of the Euler-Maruyama approximation. The proof of the result is based on an approximation argument introduce in \cite{Komatsu} and which is very close to the Yamada-Watanabe approximation technique.

In the recent years, there has been a lot of studies on the rate of convergence for the Euler-Maruyama approximation with non-Lipschitz coefficients.
In the diffusion case, assuming that the drift coefficient is the sum of a Lipschitz continuous function and monotone decreasing H\"older continuous function, \cite{Gyongy} gives the order of the strong rate of convergence for one-dimensional SDEs.
This result was generalized in \cite{NT} in $d$-dimension and by considering the class of one-sided Lipschitz drift coefficients.
The proof of results in \cite{Gyongy, NT} are based on a Yamada-Watanabe approximation technique (see \cite{YaWa}).
In the jump-diffusion case, the work \cite{Qiao} shows that when the coefficients are non-Lipschitz, the Euler-Maruyama approximation has a strong convergence. The concept used to prove their result is to the Yamada-Watanabe approach and is based on a  generalized Gronwall inequality.

Let us also mention the work \cite{MP}, where the author prove that if $L$ is a Wiener process and the coefficients are $\beta$-H\"older continuous, then the Euler-Maruyama approximation converges weakly to the unique weak solution of the corresponding SDE with rate $\beta/2$.
This result was extended to the case of nondegenerate SDEs driven by L\'evy processes in \cite{MiZh}. More specifically, the authors study the dependence of the rate on the regularity of the coefficient and the driving noise. The backward Kolmogorov equation plays a crucial role in their argument.

In this paper, we first to study the strong rate of convergence of the Euler-Maruyama scheme \eqref{EM_1} when $L$ is a Wiener process and the drift coefficient $b$ is singular, that is $b$ is bounded and $\beta$-H\"older continuous in space variable. Our method to study the strong rate of convergence differs from the existing ones.
We develop a method based on the regularity of the solution to the Kolmogorov equation associated to the SDE \eqref{SDE_1}.
More precisely, using It\^o's formula, we write the drift part in terms of the solution to the Kolmogorov equation.
Applying the estimates of the derivatives of the solution to the Kolmogorov equation, we are able to obtain the convergence rate (see Theorem \ref{main_1.5}).

Second, we examine the Euler-Maruyama scheme \eqref{EM_1} when the driving noise is a truncated symmetric $\alpha$-stable process with $\alpha \in (1,2)$.
The method is also based on the regularity of the solution to the Kolmogorov equation associated to the SDE \eqref{SDE_1}. The authors are not aware of any other work where the regularity of the Kolmogorov equation is used to study strong convergence of Euler-Maruyama approximation with irregular coefficient. Traditionally, the Kolmogorov equation based method is used to study the weak convergence rate for the Euler-Maruyama approximation (see for example \cite{ProTal}  and references therein) which is for example very important in financial applications.

This paper is divided as follows:
In Section \ref{sec:results}, we introduce some notations and preliminary results on existence and regularity of the Kolmogorov backward equation associated to the SDE \eqref{SDE_1}. The main results of this paper are also given in this section.
Section \ref{sec:proof} is devoted to the proof of the main results.
%In Section \ref{sec:appendix}, we provide moment estimate and It\^o's formula for symmetric $\alpha$-stable processes.

\section{Preliminaries}\label{sec:results}

\subsection{Notations}

Let $\mathcal{B}(\real^d_0)$ be the Borel-$\sigma$-algebra on $\real^d_0$, with $\real^d_0:=\real^d \setminus \{0\}$ and set $\nabla \equiv D =(\frac{\partial}{\partial x_1},\ldots,\frac{\partial}{\partial x_d} )^\ast$, $D^2=(\frac{\partial^2}{\partial x_i x_j})_{1 \leq i,j \leq d}$ and $\Delta=\sum_{i=1}^d \frac{\partial^2}{\partial x_i^2}$. Here $^\ast$ is the transpose of a vector or matrix. In the following we introduce some space of function:
\begin{itemize}
	\item $C_b(\real^d;\real^k),\,d,k \in \n$ denotes the space of bounded continuous functions from $\real^d $ to $\real^k$. In particular, if $k=1$, we say $C_b(\real^d;\real)=C_b(\real^d)$. For bounded measurable function $f$, the supremum norm of $f$ is defined by $\|f\|_{\infty}:=\sup_{x\in \real^d}|f(x)|$.
	
	\item $C_c^{\infty}(\real^d)$ denotes the space of all infinitely differentiable $\real$-valued functions with compact support contained in $\real$.
	
	\item $C_b^{\beta}(\real^d;\real^k),\,\beta \in (0,1)$ denotes the set of all functions from $\real^d $ to $\real^k$ which are bounded and $\beta$-H\"older continuous functions. Hence if $f \in C_b^{\beta}(\real^d;\real^k)$, then
	\begin{align*}%[f]_{\beta}:=
	\sup_{x,y \in \real^d, x \neq y} \frac{|f(x)-f(y)|}{|x-y|^{\beta}}<\infty.
	\end{align*}
	
	\item $C_b^{i,\beta}(\real^d),\,i=1,2$ and $\beta \in (0,1)$ denotes the space of $i$-times differentiable functions $f:\real^d \to \real$ with $D^{\ell}f \in C_b^{\beta}(\real^d; \real^{\otimes\ell})$ for any $1 \leq \ell \leq i$. A function $f:\real^d \to \real^d$ belongs to $C_b^{i,\beta}(\real^d;\real^d)$ if each it components $f_j$ is in $C_b^{i,\beta}(\real^d)$ for $j=1,\ldots,d$.
	
	\item Let $\mathbf{F}$ be a class of functions and $[a,b]$ be closed interval.  $C([a,b],\mathbf{F})$ denotes the space of all functions $f:[a,b] \times \real^d \to \real^k$ such that $f(t,\cdot) \in \mathbf{F}$ for any $t\in[a,b]$.
	
	\item $C^1([a,b],\mathbf{F})$ denotes the space of all functions $f$ such that $f \in C([a,b],\mathbf{F})$ and $ \frac{\partial f}{\partial t}(t,\cdot)$ exists, is continuous and is in $\mathbf{F}$.
	
	\item For $a<b$, we write $C_b^{\beta}([a,b])$ for $C([a,b];C_b^{\beta}(\real^d;\real^d))$ and define the norm $\|\cdot\|_{C_b^{\beta}([a,b])}$ on $C_b^{\beta}([a,b])$ by
	\begin{align*}
	\|f\|_{C_b^{\beta}([a,b])} := \sup_{v \in [a,b], x \in \real^d}|f(v,x)|+\sup_{v \in [a,b], x \neq y}\frac{|f(v,x)-f(v,y)|}{|x-y|^{\beta}}.
	\end{align*}
\end{itemize}

\subsection{L\'evy processes: definition and basic properties}

In this section, we recall the definition and basic properties of a  L\'evy process.
We also give the definition of truncated symmetric $\alpha$-stable process. For more information on L\'evy processes, we refer the reader to \cite{Applebaum, Bertoin, KimSong, Sato} and references therein.% (truncated or not).

\begin{Def}
	A stochastic process $L=(L_t)_{0\leq t \leq T}$ on $\real^d$ is called a L\'evy process on probability space $(\Omega, \mathcal{F},\p)$ if the following conditions are satisfied:
\begin{itemize}
	\item[(i)]  For any choice of $n \in \mathbb{N}$ and $0 \leq t_1 < \cdots <t_n$, the random variables $L_{t_0}, L_{t_1}-L_{t_0}, \ldots, L_{t_n}-L_{t_{n-1}}$ are independent.
	\item[(ii)]  $L_0=0$, a.s..
	\item[(iii)] For any $t,s \geq 0$, the distribution of $L_{s+t}-L_{s}$ does not depend on $s$.
	\item[(iv)]  For any $\varepsilon>0$ and $t \geq 0$, $\lim_{s\rightarrow t}\p(|L_{s}-L_{t}|>\varepsilon)=0$.
	\item[(v)] There exists $\Omega _0 \in \mathcal{F}$ with $\p(\Omega _0) =1$ such that, for every $\omega \in \Omega_0 $, $L_t(\omega )$ is right-continuous in $t \geq 0$ and has left limits in $t>0$, i.e. $L_t(\omega )$ is c\`adl\`ag function on $[0, T]$.
\end{itemize}
\end{Def}

Next, we define a Poisson random measure associated to a L\'evy process $L$.
Let $\Delta L_t:=L_t-L_{t-}$ be the jump size of $L$ at time $t$.
We define a \textsl{Poisson random measure} for $L$ on $\mathcal{B}([0,\infty)) \times \mathcal{B}(\real^d_0) $ by
\begin{align*}
N(t,F):=\sum_{0 \leq s \leq t} 1_F(\Delta L_s),~F\in \mathcal{B}(\real^d_0).
\end{align*}
Then the \textsl{L\'evy measure} $\nu$ is defined by $\nu(F)=\e[N(1,F)]$ for $F\in \mathcal{B}(\real^d_0)$.

By the L\'evy-It\^o decomposition (see for example \cite[Theorem 19.2]{Sato}), a L\'evy process $L$ admits the following integral representation:
\begin{align*}%\label{deco:1}
L_t
=bt
+\sigma W_t
+\int_0^t \int_{|z| \leq 1} z \widetilde{N}(\diffns s,\diffns z)+\int_0^t \int_{|z| > 1} z N(\diffns s,\diffns z),
\end{align*}
where $b \in \real^d$, $\sigma \in \real^{d} \times \real^d$, $W=(W_t)_{0\leq t \leq T}$ is a $d$-dimensional Wiener process and  $\widetilde{N}$ defined by
\begin{align*}
\widetilde{N}(\diffns s,\diffns z)=N(\diffns s,\diffns z)-\nu(\diffns z)\diffns s,
\end{align*}
 is the \textsl{compensated Poisson random measure} of $L$.
\begin{Rem}\label{important}
	Let $L=(L_t)_{0\leq t \leq 1}$ be a $d$-dimensional L\'evy process with $\e[|L_t|^{p}]<\infty$ for some $p \geq 2$.
	Then it follows from \cite[Theorem 66]{Protter} that, there exists $C_p>0$ such that for all $t \in [0,1]$, we have
	\begin{align*}
		\e\left[\sup_{0\leq s \leq t}|L_s|^p \right] \leq C_p t.
	\end{align*}
	%For any $p \geq 2$, there exists positive constant $C$ such that for any $t \in [0,1]$, we have
	%\begin{align*}
	%\e[\sup_{0 \leq s \leq t}|L_s|^p] \leq C t.
	%\end{align*}
	%	\textcolor{red}{Example: Let $N_t$ be a Poisson process with intensity $\lambda$. For any $p \geq 2$, $\e[N_t^2] =(\lambda t)^2+\lambda t, \e[N_t^3]=(\lambda t)^3+3(\lambda t)^2+\lambda t$.}
\end{Rem}

A $d$-dimensional ($d \geq 2$) \textsl{truncated symmetric $\alpha$-stable process} $L$ with $\alpha \in (1,2)$ is a L\'evy process such that
\begin{align*}
	\e[e^{i \langle \xi, L_t \rangle}]
	=e^{-t \psi(\xi)},
	~~\xi \in \real^d.
\end{align*}
Here the function $\psi$ is the characteristic exponent of $L$ and is given by
\begin{align*}
	\psi(\xi)
	=\mathcal{A}(d,-\alpha) \int_{|z| \leq 1} \left\{1-\cos(\langle \xi,z \rangle) \right\} \nu(\diffns z),~~
\end{align*}
with
$$
	\mathcal{A}(d,-\alpha):=\frac{\alpha 2^{\alpha-1} \Gamma((d+\alpha)/2)}{\pi^{d/2} \Gamma(1-\alpha/2)}
$$
and the L\'evy measure given by
\begin{align*}
	\nu(\diffns z)
	=\frac{1(|z|\leq 1)}{|z|^{d+\alpha}} \diffns z.
\end{align*}

It follows from the L\'evy-It\^o decomposition, a truncated symmetric $\alpha$-stable process $L$ admits the following integral representation:
\begin{align}\label{deco:1}
	L_t=\int_0^t \int_{|z|\leq 1} z \widetilde{N}(\diffns s,\diffns z).
\end{align}
Note that for any $n \in \n$, $\int_{|z| \geq 1} |z|^n \nu(dz)=0$. Then using \cite[Theorem 2.5.2]{Applebaum}, the $n$-th moment of $L_t$ is finite for any $t\geq 0$. In particular, the second moment is finite. This means that a truncated symmetric $\alpha$-stable process $L$ is a square integrable L\'evy process.
By the integral representation \eqref{deco:1} and the It\^o isometry, we conclude
\begin{align*}%\label{L2:truncated}
	\e[|L_t|^2]
	=\int_0^t \int_{|z| \leq 1} |z|^2 \nu(\diffns z) \diffns s
	=t \int_{|z| \leq 1} \frac{1}{|z|^{d+\alpha-2}} \diffns y
	= ct,
\end{align*}
for some constant $c$.

Finally the infinitesimal generator $\mathcal{L}$ of a truncated symmetric $\alpha$-stable $L$ is given by
\begin{align*}
\mathcal{L}f(x):=\int_{|z|\leq 1} \big( f(x+z)-f(x) - \langle z, Df(x) \rangle \big) \nu(\diffns z),~f\in C_c^{\infty}(\real^d).
\end{align*}

\subsection{Results on existence of solutions to the Kolmogorov equations}

In this section, we present some auxiliary results on existence and regularity of solutions to the Kolmogorov equations. Further, we derive some explicit expressions of the constants in the upper bounds of the derivative of these solutions.
These results will play a crucial role in the proof of our main theorems.

%The following Theorems (Theorem \ref{thm:1}, Theorem \ref{thm:stable} and Theorem \ref{thm:2}) provide the smoothnesss of solutions to the Kolmogorov equations.

%The two following Theorems which are due to \cite[Theorem 2.8]{Flandoli} and \cite[Theorem 3.4]{Priola} provide the smoothness of the solutions of the Kolmogorov equations.
%\cite[Theorem 3.4]{Priola}

\subsubsection{Case of Wiener process}

The following result corresponds to \cite[Theorem 2.8]{Flandoli} provides regularity of solution to the Kolmogorov equation associated to the SDE \eqref{SDE_1} when the driving process is a Wiener process.

\begin{Thm}\label{thm:1}
Assume that $L=W$ be a $d$-dimensional Wiener process.
Let $t_1 \in [0,T]$.
For all $\varphi \in C([0,t_1];C_b^{\beta}(\real^d;\real^d))$, there exists at least one solution $u$ to the backward Kolmogorov equation
\begin{align}\label{Kolmo_wiener}
\frac{\partial u}{\partial t}  + \nabla u \cdot b + \frac{1}{2} \Delta u=-\varphi \text{ on } [0,t_1] \times \real^d,~u(t_1,x)=0
\end{align}
of class
\begin{align*}
u \in C([0,t_1];C_b^{2,\beta'}(\real^d;\real^d)) \cap C^1([0,t_1];C_b^{\beta'}(\real^d;\real^d))
\end{align*}
for all $\beta' \in (0,\beta)$ with
\begin{align*}
\|D^2u\|_{C_b^{\beta'}([0,t_1])} \leq C_{\beta'}\|\varphi\|_{C_b^{\beta}([0,t_1	])}
\end{align*}
and
\begin{align*}
\|\nabla u\|_{C_b^{\beta}([0,t_1])} \leq C(t_1) \|\varphi\|_{C_b^{\beta}([0,t_1])} \text{ with } \lim_{t_1 \to 0} C(t_1)=0.
\end{align*}
\end{Thm}

%\begin{Rem}
%	One can show that $C(t)=C_0 \cdot (t^{\frac{1+\beta}{2}} \vee t^{1/2} )$ for some constant $C_0>0$ (see \cite[Theorem 2.3 and 2.8]{Flandoli}).
%\end{Rem}
In the following Lemma we make precise the function $C(t_1)$ derived in Theorem \ref{thm:1}.
\begin{Lem}\label{const_wiener}
	The constant $C(t_1)$ defined in Theorem \ref{thm:1} is given by $C(t_1)=C_0 t_1^{1/2}$ for some constant $C_0$ and $t_1 \in (0,T]$ with $\|b\|_{C_b^{\beta}([0,T])} C_0 t_1^{1/2} \leq 1/4$.
\end{Lem}
\begin{proof}
	Let us first prove that if $v$ is the solution to the heat equation
	\begin{align}\label{heat_eq_1}
	\frac{\partial v}{\partial t}  = \frac{1}{2} \Delta v+\varphi \text{ on } [0,t_1] \times \real^d,~v(0,x)=0,
	\end{align}
	with $\varphi \in C([0,t_1];C_b^{\beta}(\real^d;\real^d))$, then it holds that
	\begin{align}\label{heat_eq_2}
	\|\nabla v\|_{C_b^{\beta}([0,t_1])}
	\leq C(t) \|\varphi\|_{C_b^{\beta}([0,t_1])}.
	\end{align}
It is shown in \cite[Theorem 2.3]{Flandoli} that
	\begin{align*}
	v(t,x)=\int_{0}^{t} \e[\varphi(s,x+W_{t-s})]\diffns s, t \in [0,t_1]
	\end{align*}
	is a solution to the equation \eqref{heat_eq_1} and
	\begin{align*}
	\frac{\partial}{\partial x_i} \e[\varphi(s,x+W_{t-s})]
	=-\frac{1}{t-s} \e[\varphi(s,x+W_{t-s}) W_{t-s}^i],
	%=-\frac{1}{t-s} \int_{\real^d} \varphi(s,x+z) z_i \frac{\exp\left(-\frac{|z|^2}{2(t-s)}\right)}{(2\pi (t-s))^{d/2}} dz.
	\end{align*}
	where $W_t^i$ is the $i$-th coordinate of $W_t$.
	Since the function $|x|\exp(-|x|^2)$ is bounded, there exists $\widetilde{C}_0>0$ such that
	\begin{align}\label{heat_eq_31}
	|\nabla v(t,x)|
	\leq \widetilde{C}_0 \sup_{s \in [0,t_1],x \in \real^d}|\varphi(s,x)| \int_{0}^{t} \frac{1}{\sqrt{t-s}}\diffns s
	\leq C_0 t_1^{1/2} \|\varphi\|_{C_b^{\beta}([0,t_1])},
	\end{align}
	where $C_0=2\widetilde{C}_0$.
	Similarly, we get
	\begin{align}\label{heat_eq_41}
	\frac{|\nabla v(t,x)-\nabla v(t,y)|}{|x-y|^{\beta}}
	\leq \widetilde{C}_0 \sup_{s \in [0,t_1],x\neq y}\frac{|\varphi(s,x)-\varphi(s,y)|}{|x-y|^{\beta}} \int_{0}^{t} \frac{1}{\sqrt{t-s}}\diffns s
	\leq C_0 t_1^{1/2} \|\varphi\|_{C_b^{\beta}([0,t_1])}.
	\end{align}
	Combining \eqref{heat_eq_31} and \eqref{heat_eq_41}, we conclude that \eqref{heat_eq_2} holds.
	
	Next, set $u^{(0)}=0$ and for $n \in \n \cup \{0\}$, let $u^{(n+1)}$ be a solution to
	\begin{align*}%\label{heat_eq_5}
	\frac{\partial u^{(n+1)}}{\partial t}  + \frac{1}{2} \Delta u^{(n+1)}=-(b\cdot \nabla)u^{(n)} - \varphi \text{ on } [0,t_1] \times \real^d,~u^{(n+1)}(t_1,x)=0.
	\end{align*}
	Define $v^{(n)}, \,n \in \n \cup \{0\}$ by $v^{(n)}:=(b\cdot\nabla)u^{(n)}+\varphi$.
	From \eqref{heat_eq_2}, we have
	$$\|\nabla u^{(n+1)}\|_{C_b^{\beta}([0,t_1])} \leq C(t_1)\|\nabla v^{(n)}\|_{C_b^{\beta}([0,t_1])}$$
	 for any $n\in\n \cup \{0\}$.
	Choose $t_1$ with $\|b\|_{C_b^{\beta}([0,T])} C(t_1) \leq 1/4$, we have
	\begin{align*}
	\|\nabla v^{(n)}\|_{C_b^{\beta}([0,t_1])}
	&\leq \|\varphi\|_{C_b^{\beta}([0,t_1])}+2\|b\|_{C_b^{\beta}([0,t_1])} \|\nabla u^{(n)}\|_{C_b^{\beta}([0,t_1])}\\
	&\leq \|\varphi\|_{C_b^{\beta}([0,t])}+2\|b\|_{C_b^{\beta}([0,T])} C(t_1) \|\nabla v^{(n-1)}\|_{C_b^{\beta}([0,t_1])}\\
	&\leq \cdots
	\leq \sum_{i=1}^{n}\frac{1}{2^{i}} \|\varphi\|_{C_b^{\beta}([0,t_1])}
	<2\|\varphi\|_{C_b^{\beta}([0,t_1])}.
	\end{align*}
	Hence, it holds that
	\begin{align*}%\label{heat_eq_6}
	\|\nabla u^{(n+1)}\|_{C_b^{\beta}([0,t_1])} \leq 2C(t_1)\|\varphi\|_{C_b^{\beta}([0,t_1])}.
	\end{align*}
	On the other hand, we know from \cite[Theorem 2.8]{Flandoli} that there exists a subsequence $(u^{(n_k)})_{k \in \n}$ of $(u^{(n)})_{n \in \n}$ that converges uniformly in $(t,x) \in [0,t_1] \times \real^d$ to some $u$ such that $u$ is a solution to the backward Kolmogorov equation \eqref{Kolmo_wiener} and $\nabla u^{(n_k)}$ converges to $\nabla u$ in $C_b^{\beta}([0,t_1])$. Hence the result follows.
\end{proof}

%\begin{Rem}
%	We can replace the assumption $\|b\|_{C_b^{\beta}([0,t])} C_0 t^{1/2} \leq 1/4$ to $\|b\|_{C_b^{\beta}([0,T])} C_0 t^{1/2} \leq 1/4$.
%\end{Rem}

Similar arguments as in the proof of Theorem \ref{thm:1} and Lemma \ref{const_wiener} lead to the following corollary:

\begin{Cor}\label{cor:1}
	Let $0 \leq t_1<t_2 \leq T$ with $\|b\|_{C_b^{\beta}([0,T])} C_0\cdot (t_2-t_1)^{1/2} \leq 1/4$.
	For all $\varphi \in C([t_1,t_2];C_b^{\beta}(\real^d;\real^d))$, there exists at least one solution $u$ to the backward Kolmogorov equation
	\begin{align*}%\label{Kolmo_wiener}
	\frac{\partial u}{\partial t}  + \nabla u \cdot b + \frac{1}{2} \Delta u=-\varphi \text{ on } [t_1,t_2] \times \real^d,~u(t_2,x)=0
	\end{align*}
	of class
	\begin{align*}
	u \in C([t_1,t_2];C_b^{2,\beta'}(\real^d;\real^d)) \cap C^1([t_1,t_2];C_b^{\beta'}(\real^d;\real^d))
	\end{align*}
	for all $\beta' \in (0,\beta)$ with
	\begin{align*}
	\|D^2u\|_{C_b^{\beta'}([t_1,t_2])} \leq M \|\varphi\|_{C_b^{\beta}([t_1,t_2])}
	\end{align*}
	for some constant $M$ and
	\begin{align*}
	\|\nabla u\|_{C_b^{\beta}([t_1,t_2])} \leq C_0(t_1-t_2)^{1/2} \|\varphi\|_{C_b^{\beta}([t_1,t_2])}
	\end{align*}
	for some constant $C_0$.
	%Here the norm $\|\cdot\|_{C_b^{\beta}([s,t])}$ is defined by
	%\begin{align*}
	%\|\varphi\|_{C_b^{\beta}([s,t])} := \sup_{v \in [s,t], x \in \real^d}|\varphi(v,x)|+\sup_{v \in [s,t], x \neq y}\frac{|\varphi(v,x)-\varphi(v,y)|}{|x-y|^{\beta}}.
	%\end{align*}
\end{Cor}

%\begin{Rem}
%In this case, $C(s,t)=C_0 \cdot ((t-s)^{\frac{1+\beta}{2}} \vee (t-s)^{1/2} )$ for some constant $C_0>0$.
%\end{Rem}
The following result will also be needed in the proof of the main results.
\begin{Lem}\label{cor:2}
	Let $T>0$.
	For any $\varepsilon \in (0,1)$, there exist $m \in \n$ and $(T_j)_{j=0, \ldots,m}$ such that $0=T_0<T_j<T_{j+1}<\cdots<T_m=T$ and for any $j=0,\ldots,m-1$,
	\begin{align*}
	\|\varphi\|_{C_b^{\beta}([0,T])} C_0\cdot (T_{j+1}-T_{j})^{1/2} \leq \varepsilon \text{ and }
	\|b\|_{C_b^{\beta}([0,T])} C_0\cdot (T_{j+1}-T_{j})^{1/2} \leq \frac{1}{4}.
	\end{align*}
\end{Lem}
\begin{proof}
	For any $\varepsilon \in (0,1)$, we define
	\begin{align*}
	\delta
	:=\left( \frac{\varepsilon}{C_0\|\varphi\|_{C_b^{\beta}([0,T])}} \wedge \frac{1}{4C_0\|b\|_{C_b^{\beta}([0,T])}} \right)^{2}
	\end{align*}
	Since $\delta>0$, there exists $m \in \n$ such that $(m-1) \delta < T \leq m \delta$.
	Define $T_0:=0$, $T_m:=T$ and $T_j:=T_{j-1}+\delta = j \delta$ for $j=1,\ldots, m-1$.
	Then for any $j =0,\ldots, m$, $T_j-T_{j-1} \leq \delta$, we have
	\begin{align*}
	\|\varphi\|_{C_b^{\beta}([0,T])} C_0\cdot (T_{j+1}-T_{j})^{1/2} \leq \varepsilon \text{ and }
	\|b\|_{C_b^{\beta}([0,T])} C_0\cdot (T_{j+1}-T_{j})^{1/2} \leq \frac{1}{4}.
	\end{align*}
	This concludes the proof.
\end{proof}

\subsubsection{Case of truncated $\alpha$-stable process}

The following result is due to \cite[Theorem 17]{HaPr} and it provides existence and regularity of the Kolmogorov equation in a given space when the driven noise is a truncated symmetric $\alpha$-stable process.

%%%%%%%%%%%%%Truncated stable%%%%%%%%%%%%%
%\begin{comment}
\begin{Thm}\label{thm:2}
Assume that $L=(L_t)_{0 \leq t \leq T}$ is a $d$-dimensional truncated-$\alpha$-stable process for $\alpha \in (1,2)$ and $d \geq 2$.
Let $t_1 \in [0,T]$.
Suppose that $\varphi \in C([0,t_1];C_b^{\beta}(\real^d;\real^d))$ for $\beta \in (0,1)$ with $\alpha+\beta >2$.
Then, there exists a $u \in C([0,t_1],C_b^2(\real^d;\real^d)) \cap C^1([0,t_1],C_b(\real^d;\real^d))$ satisfying the backward Kolmogorov equation
\begin{align}\label{Kolmo_stable}
\frac{\partial u}{\partial t}  + \nabla u \cdot b + \mathcal{L} u=-\varphi \text{ on } [0,t_1] \times \real^d,~u(t_1,x)=0,
\end{align}
with
\begin{align*}
\|\nabla u\|_{C_b^{\beta}([0,t_1])} \leq C_{\alpha}(t_1)\|\varphi\|_{C_b^{\beta}([0,t_1])} \text{ with } \lim_{t_1 \to 0} C_{\alpha}(t_1)=0,
\end{align*}
and
\begin{align*}
\|D^2 u\|_{\infty} \leq M \|\varphi\|_{C_b^{\beta}([0,t_1])},
\end{align*}
for some positive constant $M$.% and a function $C_{\alpha}(t)$ to be made precise in the next Lemma.
\end{Thm}

%\begin{Rem}
%	One can show that $C_{\alpha}(t)=C_0 \cdot t^{1-1/\alpha}$ for some constant $C_0>0$ (see \cite[Theorem 16 and 17]{HaPr}).
%	Moreover, in the proof of the above theorem, the assumption $C_{\alpha}(t)\|\varphi\|_{C_b^{\beta}} \leq 1/4$, (see page 5343) is required.
%\end{Rem}
The function $C_{\alpha}(t_1)$ is made more precise in the following lemma:

\begin{Lem}\label{const_stable}
	The function $C_{\alpha}(t_1)$ defined in Theorem 2.8 is given by $C_{\alpha}(t_1)=C_0 t_1^{1-1/\alpha}$ for some constant $C_0$ and $t_1 \in (0,T]$ with $\|b\|_{C_b^{\beta}([0,T])} C_0 t_1^{1-1/\alpha} \leq 1/4$.
\end{Lem}
\begin{proof}
	Let us first prove that the solution $v$ to the equation
	\begin{align}\label{stable_eq_1}
	\frac{\partial v}{\partial t}  = \mathcal{L} v+\varphi \text{ on } [0,t_1] \times \real^d,~v(0,x)=0,
	\end{align}
	for $\varphi \in C([0,t_1];C_b^{\beta}(\real^d;\real^d))$ satisfies
	\begin{align}\label{stable_eq_2}
	\|\nabla v\|_{C_b^{\beta}([0,t_1])}
	\leq C_{\alpha}(t_1) \|\varphi\|_{C_b^{\beta}([0,t_1])}.
	\end{align}
It follows from \cite[Theorem 16]{HaPr} that
	\begin{align*}
	v(t,x)=\int_{0}^{t} \e[\varphi(s,x+L_{t-s})]\diffns s, t \in [0,t_1]
	\end{align*}
	is a solution to the equation \eqref{stable_eq_1} and (see (19) of \cite{HaPr})
	\begin{align*}
	\frac{\partial}{\partial x_i} \e[\varphi(s,x+L_{t})]
	=-\int_{\real^d}\varphi(s,t^{1/\alpha}u+x) t^{\frac{d-1}{\alpha}} \frac{\partial}{\partial x_i}(p_t(t^{1/\alpha}u))\diffns u.
	\end{align*}
	Here $p_t(x)$ is the density function of $L_t$ which satisfies
	\begin{align*}
	\int_{\real^d} \left| \frac{\partial}{\partial x_i}(p_t(t^{1/\alpha}u)) \right| \diffns u
	\leq C_0' t^{-d/\alpha}
	\end{align*}
	for some constant $C'_0$, (see, \cite{HaPr}, page 5338).
	Hence there exists a constant $\widetilde{C}_0>0$ such that
	\begin{align}\label{stable_eq_31}
	|\nabla v(t,x)|
	\leq \widetilde{C}_0 \sup_{s \in [0,t_1],x \in \real^d}|\varphi(s,x)| \int_{0}^{t} \frac{1}{(t-s)^{1/\alpha}}\diffns s
	\leq C_0 t_1^{1-1/\alpha} \|\varphi\|_{C_b^{\beta}([0,t_1])},
	\end{align}
	where $C_0:=\alpha/(\alpha-1) \widetilde{C}_0$. Similarly
	\begin{align}\label{stable_eq_41}
	\frac{|\nabla v(t,x)-\nabla v(t,y)|}{|x-y|^{\beta}}
	\leq \widetilde{C}_0 \sup_{s \in [0,t_1],x\neq y}\frac{|\varphi(s,x)-\varphi(s,y)|}{|x-y|^{\beta}} \int_{0}^{t} \frac{1}{(t-s)^{1/\alpha}}\diffns s
	\leq C_0 t_1^{1-1/\alpha} \|\varphi\|_{C_b^{\beta}([0,t_1])}.
	\end{align}
	Combining \eqref{stable_eq_31} and \eqref{stable_eq_41}, we conclude that \eqref{stable_eq_2} holds.

	Next, set $u^{(0)}=0$ and for $n \in \n \cup \{0\}$, $u^{(n+1)}$ be a solution to
	\begin{align*}%\label{stable_eq_5}
	\frac{\partial u^{(n+1)}}{\partial t}  + \mathcal{L} u^{(n+1)}=-(b\cdot \nabla)u^{(n)} - \varphi \text{ on } [0,t_1] \times \real^d,~u^{(n+1)}(t_1,x)=0.
	\end{align*}
	Arguing as in Lemma \ref{const_wiener}, we get by
	%Let $v^{(n)}:=(b\cdot\nabla)u^{(n)}+\varphi$.
	%From \eqref{heat_eq_2}, we have $\|\nabla u^{(n+1)}\|_{C_b^{\beta}([0,t])} \leq C(t)\|\nabla v^{(n)}\|_{C_b^{\beta}([0,t])}$ for any $n\in\n$.
	%By choosing $t$ such that $\|b\|_{C_b^{\beta}([0,t])} C(t) \leq 1/4$, we have
	%\begin{align*}
	%\|\nabla v^{(n)}\|_{C_b^{\beta}([0,t])}
	%&\leq \|\varphi\|_{C_b^{\beta}([0,t])}+2\|b\|_{C_b^{\beta}([0,t])} \|\nabla u^{(n)}\|_{C_b^{\beta}([0,t])}\\
	%&\leq \|\varphi\|_{C_b^{\beta}([0,t])}+2C(t)\|b\|_{C_b^{\beta}([0,t])} \|\nabla v^{(n-1)}\|_{C_b^{\beta}([0,t])}\\
	%&\leq \cdots
	%\leq \sum_{i=1}^{n}\frac{1}{2^{i}} \|\varphi\|_{C_b^{\beta}([0,t])}
	%<2\|\varphi\|_{C_b^{\beta}([0,t])}.
	%\end{align*}
	%Hence, it holds that
	choosing $t_1$ with $\|b\|_{C_b^{\beta}([0,T])} C_{\alpha}(t_1) \leq 1/4$
	\begin{align*}%\label{stable_eq_6}
	\|\nabla u^{(n+1)}\|_{C_b^{\beta}([0,t_1])} \leq 2C_{\alpha}(t_1)\|\varphi\|_{C_b^{\beta}([0,t_1])}.
	\end{align*}
	Furthermore, we know from \cite[Theorem 17]{HaPr} that the sequence $(u^{(n)})_{n \in \n}$ which converges uniformly in $(t,x) \in [0,t_1] \times \real^d $ to some $u$ such that $u$ is a solution to the backward Kolmogorov equation \eqref{Kolmo_stable} and $\nabla u^{(n)}$ converges $\nabla u$ in $C_b^{\beta}([0,t_1])$.
	Hence the result follows.
\end{proof}

%\begin{Rem}
%	We can also replace the assumption $\|b\|_{C_b^{\beta}([0,t])} C_0 t^{1-1/\alpha} \leq 1/4$ to $\|b\|_{C_b^{\beta}([0,T])} C_0 t^{1-1/\alpha} \leq 1/4$.
%\end{Rem}

The following result can be derived using similar arguments as in the proof of Theorem \ref{thm:2} and Lemma \ref{const_stable}.

\begin{Cor}\label{cor:stable:1}
	Let $0 \leq t_1 < t_2 \leq T$ with $\|b\|_{C_b^{\beta}([0,T])} C_0\cdot (t_2-t_1)^{1-1/\alpha} \leq 1/4$.
	For all $\varphi \in C([t_1,t_2];C_b^{\beta}(\real^d;\real^d))$ for $\beta \in (0,1)$ with $\alpha+\beta >2$, there exists a $u \in C([t_1,t_2],C_b^2(\real^d;\real^d)) \cap C^1([t_1,t_2],C_b(\real^d;\real^d))$ satisfying the backward Kolmogorov equation
	\begin{align*}%\label{Kolmo_stable_2}
	\frac{\partial u}{\partial t}  + \nabla u \cdot b + \mathcal{L} u=-\varphi \text{ on } [t_1,t_2] \times \real^d,~u(t_2,x)=0,
	\end{align*}
	with
	\begin{align*}
	\|\nabla u\|_{C_b^{\beta}([t_1,t_2])} \leq C_0 (t_2-t_1)^{1-1/\alpha}\|\varphi\|_{C_b^{\beta}([t_1,t_2])}
	\end{align*}
and
	\begin{align*}
	\|D^2 u\|_{\infty} \leq M \|\varphi\|_{C_b^{\beta}([t_1,t_2])},
	\end{align*}
	for some constant $M$.
\end{Cor}

%\begin{Rem}
%	In this case, $C_{\alpha}(s,t)=C_0 (t-s)^{1-1/\alpha}$ for some constant $C_0$.
%\end{Rem}

We also have

\begin{Lem}\label{cor:stable:2}
	Let $T >0$.
	For any $\varepsilon \in (0,1)$, there exist $m \in \n$ and $(T_j)_{j=0, \ldots,m}$ such that $0=T_0<T_j<T_{j+1}<\cdots<T_m=T$ and for any $j=0,\ldots,m-1$,
	\begin{align*}
	\|\varphi\|_{C_b^{\beta}([0,T])} C_0 \cdot (T_{j+1}-T_{j})^{1-1/\alpha} \leq \varepsilon
	\text{ and }
	\|b\|_{C_b^{\beta}([0,T])} C_0 \cdot (T_{j+1}-T_{j})^{1-1/\alpha} \leq \frac{1}{4}.
	\end{align*}	
\end{Lem}

\subsection{Main theorems}

In this section, we state the main theorems of this paper.
We obtain results on the rates of the Euler-Maruyama approximation in $L^p$-sup norm for $p \geq 1$ if $L$ is a Wiener process or a truncated symmetric $\alpha$-stable process.% and in $L^1$-norm if $L$ is a symmetric $\alpha$-stable process.

\begin{Thm}\label{main_1.5}
Let $L=W$ be a $d$-dimensional Wiener process.
Assume that the drift coefficient $b$ is bounded $\beta$-H\"older continuous with $\beta \in (0,1)$ in space and $\eta$-H\"older continuous in time with $\eta \in [1/2,1]$, i.e., there exists $K>0$ such that for any $x,y \in \real^d$ and $t,s \in [0,T]$,
\begin{align*}
|b(t,x)-b(t,y)| \leq K|x-y|^{\beta} \text{ and }
|b(t,x)-b(s,x)| \leq K|t-s|^{\eta}.
\end{align*}
Then for any $p\geq 1$, there exists a positive constant $C$ depending on $K,M,T,d,p,x_0, \beta, \eta$ and $\|b\|_{C_b^{\beta}([0,T])}$ such that
\begin{align*}
\e\left[\sup_{0 \leq t \leq T}\left|X_t-X_t^{(n)}\right|^p \right] \leq \frac{C}{n^{p\beta/2}}.
\end{align*}
\end{Thm}

\begin{Rem}
	Under conditions of Theorem \ref{main_1.5}, the pathwise uniqueness holds for the SDE \eqref{SDE_1} (see , \cite[Corollary 2.3]{Flandoli}).
	Then \cite[Theorem D]{KaNa} guaranties that the Euler-Maruyama approximation \eqref{EM_1} converges to the corresponding SDE in $L^2$-sup norm.
	Hence, Theorem \ref{main_1.5} generalizes \cite[Theorem D]{KaNa} in two directions: first, it gives an $L^p$-sup convergence for any $p \geq 1$ and second, it gives the rate of convergence which was not given previously.
\end{Rem}

%%%%Symmetric stable%%%%%
\begin{comment}
\begin{Thm}\label{Main:stable}
	Assume that $L$ is a $d$-dimensional symmetric $\alpha$-stable process with $\alpha \in (1,2)$.
	Moreover, assume that $b \in C_b^{\beta}(\real^d;\real^d)$ with $\beta \in (0,1)$, $1<\alpha+\beta < 2$ and $ \alpha+2\beta >2$, and independent from time variable $t$.
	Then there exists a positive constant $C$ depending on $K, M, T, d, x_0, \alpha, \beta$ and $\|b\|_{C_b^{\beta}([0,T])}$ such that
	%Then for any $p \in [1,\alpha)$, there exists a positive constant $C$ depends on $K,T,x_0,d, \alpha$ and $\beta$ such that
	\begin{align*}
	\sup_{0\leq t \leq T} \e\left[\left|X_t-X_t^{(n)}\right|\right] \leq \frac{C}{n^{\beta/\alpha}}.
	\end{align*}
\end{Thm}

\begin{Rem}
	Under the assumptions of Theorem \ref{Main:stable}, it was shown that the SDE \eqref{SDE_1} has a unique strong solution (see \cite[Theorem 4.5]{Priola}).
	We shall only deal with $L^1$-convergence.
	The reason being that $\alpha$-stable processes are not $L^p$-integrable for $p \geq \alpha$.
\end{Rem}
\end{comment}

%%%%%%%%%%%%%Truncated stable%%%%%%%%%%%%%
%\begin{comment}
\begin{Thm}\label{main_2}
Assume that $L$ is a $d$-dimensional truncated symmetric $\alpha$-stable process with $\alpha \in (1,2)$ and $d \geq 2$.
Suppose that SDE \eqref{SDE_1} has a unique strong solution.
Moreover, assume that the drift coefficient $b$ is bounded, $\beta$-H\"older continuous in space with $\beta \in (0,1)$ and $\alpha+\beta>2$, and $\eta$-H\"older continuous in time with $\eta \in [1/2,1]$.
%Suppose $C_{\alpha}([0,T])\|b\|_{C_b^{\beta}([0,T])} \leq 1/4$.
Then, for any $p \geq 1$, there exists a positive constant $C$ depending on $K, M, T, d, p, x_0, \alpha, \beta, \eta$ and $\|b\|_{C_b^{\beta}([0,T])}$ such that
\begin{align*}
\e\left[\sup_{0 \leq t \leq T}\left|X_t-X_t^{(n)}\right|^p\right]
\leq \left\{ \begin{array}{ll}
\displaystyle \frac{C}{n} &\text{ if } p \beta \geq 2,  \\
\displaystyle \frac{C}{n^{p\beta/2}} &\text{ if } p \geq 2,~1 \leq 	p \beta < 2 \text{ or } p \in [1,2).
\end{array}\right.
\end{align*}
\end{Thm}
%\textcolor{red}{NOTE: In section 2.2, we prove $\e[|L_t|^2]^{1/2}=ct^{1/2}$, so when $p \leq 2$, the above result maybe true. If we can prove $\e[|L_t|^p]^{1/p} \leq C t^{1/2}$, then we may prove above result for any $p \geq 1$. 2015/6/9}

\begin{Rem}
Let $L$ be a $d$-dimensional truncated symmetric $\alpha$-stable process with $\alpha \in (1,2)$.
Choose the drift coefficient $b$ of the form
\begin{align*}
	b(t,x)=\sum_{i=1}^{m}f_i(t)b_i(x),
\end{align*}
where $f_i$ are continuous functions and $b_i$ are as in Theorem \ref{thm:2} for $i=1,\ldots,m$.
Then the SDE \eqref{SDE_1} has a unique strong solution on small interval $[0,T]$ (see \cite[Theorem 23 and Remark 24]{HaPr}).
In Remark \ref{important}, we notice that for all $L^p$-integrable L\'evy processes for some $p \geq 2$, the $p$-th moment is always bounded by $C_p t$.
Therefore, the rate of $L^p$-convergence with $p \geq 2$ and $p \beta \geq 2$ coincides.
\end{Rem}
%\end{comment}
%%%%%%%%%%%%%Truncated stable%%%%%%%%%%%%%

%\begin{Rem}
%	Let us mention that, in the symmetric $\alpha$-stable case for $\alpha \in (1,2)$, although there exists a smooth solution to the associated Kolmogorov equation (confer \cite[Theorem 3.4]{Priola}), we could not get $L^p$-convergence for the Euler-Maruyama scheme for any $p \geq 1$.
	
%	For $p \geq \alpha$, symmetric $\alpha$-stable processes do not have the $p$-th moment and hence an estimate such as those in Lemma \ref{lem:stand} is not valid.

%	For $p \in [1,\alpha)$, using It\^o's formula and the Kolmogorov equation associated to equation \eqref{SDE_1}, the difference $|X_t-X_t^{(n)}|$ can be written in terms of a solution to the Kolmogorov equation.	
%	Applying a Kunita-type inequality (see, \cite[Corollary 2.14]{Hausenblas}) and using \cite[Lemma 4.1 and Theorem 3.4]{Priola}, one obtains that the martingale part of the difference $|X_t-X_t^{(n)}|$ is $L^p$-integrable if and only if $\alpha<p$.
%	Hence contradicting $p<\alpha$.
%\end{Rem}

\section{Proof of main theorems}\label{sec:proof}

This section is devoted to the proof of the main results. The constants $C_1, C_2, C_3$ and $C$ are assumed to be positive and independent of $n$.
Unless explicitly stated otherwise, the constants $C_1, C_2, C_3$ and $C$ depend only on $K, M, T, d, p, x_0, \alpha, \beta, \lambda, \eta$ and $\|b\|_{C_b^{\beta}([0,T]})$.
Moreover, the constant $C$ may change from line to line.

We define
\begin{align*}
X_t:=(X_t^1, \ldots, X_t^d)^\ast
\text{ and }
X_t^{(n)}:=(X_t^{(1,n)},\ldots,X_t^{(d,n)})^\ast.
\end{align*}

The following estimation is standard.
For the convenience of the reader, we will give a proof.

\begin{Lem}\label{lem:stand}
Suppose that the drift coefficient $b$ is bounded and measurable.

\begin{itemize}
	\item[(i)]	If $L=W$ is a $d$-dimensional Wiener process, then for any $p >0$, there exists $C>0$ such that for any $t \in [0,T]$,
	\begin{align*}
	\e\left[\left|X_t^{(n)}-X_{\eta_n(t)}^{(n)}\right|^p\right]
	\leq
	\frac{C}{n^{p/2}}.
	\end{align*}
	
	%\item[(ii)]	If $L$ is a $d$-dimensional symmetric $\alpha$-stable process with $\alpha \in (1,2)$, then for any $p <\alpha$, there exists $C>0$ such that for any $t \in [0,T]$,
	%\begin{align*}
	%\e\left[\left|X_t^{(n)}-X_{\eta_n(t)}^{(n)}\right|^p\right]
	%\leq
	%\frac{C}{n^{p/\alpha}}.
	%\end{align*}
		
	\item[(ii)] If $L$ is a $d$-dimensional truncated symmetric $\alpha$-stable process with $\alpha \in (1,2)$ and $d \geq 2$, then for any $p >0$, there exists $C>0$ such that for any $t \in [0,T]$,
	\begin{align*}
	\e\left[\left|X_t^{(n)}-X_{\eta_n(t)}^{(n)}\right|^p\right]
	\leq
	\left\{ \begin{array}{ll}
	\displaystyle \frac{C}{n} &\text{ if } p \geq 2,  \\
	\displaystyle \frac{C}{n^{p/2}} &\text{ if } p <2.
	\end{array}\right.
	\end{align*}
	
\end{itemize}
\end{Lem}

\begin{proof}
From the definition of the Euler-Maruyama approximation, it holds from boundedness of $b$ that
\begin{align*}
|X_t^{(n)}-X_{\eta_n(t)}^{(n)}|
\leq \frac{T\|b\|_{\infty}}{n}
+|L_t-L_{\eta_n(t)}|.
\end{align*}

Suppose $L=W$.
Then for any $p \geq 1$,
\begin{align*}
\e\left[\left|X_t^{(n)}-X_{\eta_n(t)}^{(n)}\right|^p\right]
\leq \frac{2^{p-1} T^p \|b\|_{\infty}^p}{n^p}
+2^{p-1}\e\left[\left|W_t-W_{\eta_n(t)}\right|^p\right]
\leq \frac{C}{n^{p/2}},
\end{align*}
Hence the statement is true in the case of Wiener process.

%Proof of (ii).
%We suppose that $L$ is a symmetric $\alpha$-stable process with $\alpha \in (1,2)$.
%Then from Lemma \ref{moment_esti}, for any $p \in [1,\alpha)$, we have
%\begin{align*}
%\e[|L_t-L_{\eta_n(t)}|^p]
%=\e[|L_{t-\eta_n(t)}|^p]
%\leq \frac{C}{n^{p/\alpha}}.
%\end{align*}
%This concludes the proof of statement for the case of symmetric stable process.

Suppose that $L$ is a truncated symmetric $\alpha$-stable process.
It is enough to prove the statement for $p \geq 2$.
Since $\e[|L_t|^{p}]<\infty$ it follows from \cite[Theorem 66]{Protter} that there exists $C$ such that
\begin{align*}
\e[|L_t-L_{\eta_n(t)}|^p]
=\e[|L_{t-\eta_n(t)}|^p]
&\leq \frac{C}{n}.
\end{align*}
Hence the result follows when $L$ is a truncated symmetric $\alpha$-stable process.
\end{proof}

\subsection{Proof of Theorem \ref{main_1.5}}
%\subsection{Case of Wiener process}

%We first prove the rate of $L^p$-convergence for all $t \in [0,T]$.
%\begin{Lem}\label{main_wiener_1.1}
%Let $L=W$ be a $d$-dimensional Wiener process.
%We assume that the drift coefficient $b$ is a $\beta$-H\"older continuous with $\beta \in (0,1]$ in space and $\eta$-H\"older continuous with $\eta \geq 1/2$ in time.
%Suppose that conditions of Theorem \ref{main_1.5} hold.
%Then for any $p\geq 1$, there exists a positive constant $C$ depends on $K,M,T,x_0,d, \beta$ and $\eta$ such that for any $t \in [0,T]$,
%\begin{align}\label{Lp:esti}
%\e\left[\left|X_t-X_t^{(n)}\right|^p\right] \leq \frac{C}{n^{p\beta/2}}.
%\end{align}
%\end{Lem}
%\begin{proof}
For a given $\varepsilon \in (0,1)$, we consider the partition  $(T_j)_{j=0,\ldots,m}$ of closed interval $[0,T]$ which is considered in Lemma \ref{cor:2}.
For $i=1,\ldots,d$ and $j=1,\ldots,m$, Corollary \ref{cor:1} implies that there exists at least one solution $u_{i,j}$ to the backward Kolmogorov equation:
\begin{align*}%\label{Kolmo_wiener:}
\frac{\partial u_{i,j}}{\partial t}  + \nabla u_{i,j} \cdot b + \frac{1}{2} \Delta u_{i,j}=-b_i \text{ on } [T_{j-1},T_{j}] \times \real^d,~ u_{i,j}(T_{j},x)=0
\end{align*}
and $u_{i,j}$ satisfies,
\begin{align*}
\|\nabla u_{i,j}\|_{C_b^{\beta}[T_{j-1},T_j]}
\leq C_0\cdot(T_j-T_{j-1})^{1/2}\|b_j\|_{C_b^{\beta}([T_{j-1},T_j])}
\leq C_0\cdot(T_j-T_{j-1})^{1/2}\|b\|_{C_b^{\beta}([0,T])}
\leq \varepsilon.
\end{align*}
For any $t \in [T_{j-1},T_{j}]$ by  It\^o's formula, we have
\begin{align*}
u_{i,j}(t,X_t)
&=
u_{i,j}(T_{j-1},X_{T_{j-1}})
+\int_{T_{j-1}}^t \frac{\partial u_{i,j}}{\partial t} (s,X_s) \diffns s
+\int_{T_{j-1}}^t \nabla u_{i,j} (s,X_s) \diffns X_s
+\frac{1}{2} \int_{T_{j-1}}^t \Delta u_{i,j} (s,X_s) \diffns s \\
&=u_{i,j}(T_{j-1},X_{T_{j-1}})
-\int_{T_{j-1}}^t b_i(s,X_s) \diffns s
+\int_{T_{j-1}}^t \nabla u_{i,j}(s,X_s) \diffns W_s.
\end{align*}
Hence we have
\begin{align}\label{rep:wiener2}
\int_{T_{j-1}}^t b_i(s,X_s) \diffns s
=u_{i,j}(T_{j-1},X_{T_{j-1}})
-u_{i,j}(t,X_t)
+\int_{T_{j-1}}^t \nabla u_{i,j}(s,X_s) \diffns W_s.
\end{align}
In the same way, we have
\begin{align}\label{rep:wiener_EM2}
\int_{T_{j-1}}^t b_i(s,X_s^{(n)}) \diffns s
&=u_{i,j}(T_{j-1},X_{T_{j-1}}^{(n)})
-u_{i,j}(t,X_t^{(n)})
+\int_{T_{j-1}}^t \nabla u_{i,j}(s,X_s^{(n)}) \diffns W_s \nonumber\\
&+\int_{T_{j-1}}^t \nabla u_{i,j}(s,X_s^{(n)}) \cdot \left( b_i(\eta_n(s), X_{\eta_n(s)}^{(n)})-b_i(s, X_s^{(n)}) \right) \diffns s.
\end{align}
It follows from \eqref{rep:wiener2} and \eqref{rep:wiener_EM2} that for any $i=1,\ldots, d$,
\begin{align*}
X_t^i-X_t^{(n,i)}
&=X_{T_{j-1}}^i-X_{T_{j-1}}^{(n,i)}+\int_{T_{j-1}}^t \left( b_i(s,X_s)-
b_i(\eta_n(s),X_{\eta_n(s)}^{(n)})\right)\diffns s\\
&=X_{T_{j-1}}^i-X_{T_{j-1}}^{(n,i)}\\
&+\left( u_{i,j}(T_{j-1},X_{T_{j-1}})-u_{i,j}(T_{j-1},X_{T_{j-1}}^{(n)}) \right)
- \left( u_{i,j}(t,X_{t})-u_{i,j}(t,X_t^{(n)}) \right) \\
&+\int_{T_{j-1}}^t \left( \nabla u_{i,j}(s, X_s) - \nabla u_{i,j}(s,X_s^{(n)}) \right) \diffns W_s \\
&+\int_{T_{j-1}}^t \nabla u_{i,j}(s,X_s^{(n)}) \cdot \left( b(s, X_{s}^{(n)})-b(\eta_n(s), X_{\eta_n(s)}^{(n)}) \right) \diffns s\\
&+\int_{T_{j-1}}^t \left( b_i(s, X_s^{(n)}) - b_i(\eta_n(s), X_{\eta_n(s)}^{(n)}) \right)\diffns s.
\end{align*}
Since $\|\nabla u_{i,j}\|_{C_b^{\beta}([T_{j-1},T_j])} \leq \varepsilon$, by the mean-value theorem, we have
\begin{align*}
\left|X_t^i-X_t^{(n,i)}\right|
&\leq
\left|X_{T_{j-1}}^i-X_{T_{j-1}}^{(n,i)}\right| \\
&+\left| u_{i,j}(T_{j-1},X_{T_{j-1}})-u_{i,j}(T_{j-1},X_{T_{j-1}}^{(n)}) \right|
+ \left| u_{i,j}(t,X_{t})-u_{i,j}(t,X_t^{(n)}) \right| \\
&+ \left| \int_{T_{j-1}}^t \left( \nabla u_{i,j}(s, X_s) - \nabla u_{i,j}(s,X_s^{(n)}) \right) \diffns W_s \right| \\
&+\int_{T_{j-1}}^t \left| \nabla u_{i,j}(s,X_s^{(n)}) \right| \left| b(s, X_{s}^{(n)})-b(\eta_n(s), X_{\eta_n(s)}^{(n)}) \right| \diffns s\\
&+\int_{T_{j-1}}^t \left| b_i(s, X_s^{(n)}) - b_i(\eta_n(s), X_{\eta_n(s)}^{(n)}) \right|\diffns s\\
&\leq
(1+\varepsilon) \left|X_{T_{j-1}}-X_{T_{j-1}}^{(n)}\right|
+ \varepsilon \left|X_{t}-X_{t}^{(n)}\right|
+\left| \int_{T_{j-1}}^t \left( \nabla u_{i,j}(s, X_s) - \nabla u_{i,j}(s,X_s^{(n)}) \right) \diffns W_s \right| \\
&+ ( 1+ \varepsilon ) K \int_{T_{j-1}}^t \left| X_{s}^{(n)}-X_{\eta_n(s)}^{(n)} \right|^{\beta} \diffns s
+ (1+\varepsilon) K (t-T_{j-1}) \left( \frac{T}{n} \right)^{\eta}.
\end{align*}
For $p \geq 2$, using Jensen's and H\"older inequalities, we have
\begin{align*}
\left|X_t-X_t^{(n)}\right|^p
&= \left( \sum_{i=1}^d \left|X_t^{i}-X_t^{(n,i)}\right|^2 \right)^{p/2}
\leq d^{p/2-1} \sum_{i=1}^d |X_t^{i}-X_t^{(n,i)}|^p\\
&\leq
d^{p/2} 5^{p-1} (1+\varepsilon)^p \left|X_{T_{j-1}}-X_{T_{j-1}}^{(n)}\right|^p
+ d^{p/2} 5^{p-1} \varepsilon^p \left|X_{t}-X_{t}^{(n)}\right|^p \\
&+ d^{p/2-1} 5^{p-1} \sum_{i=1}^d \left| \int_{T_{j-1}}^t \left( \nabla u_{i,j}(s, X_s) - \nabla u_{i,j}(s,X_s^{(n)}) \right) \diffns W_s \right|^p \\
&+ d^{p/2} 5^{p-1} ( 1+ \varepsilon )^p K^p (t-T_{j-1})^{p-1} \int_{T_{j-1}}^t \left| X_{s}^{(n)}-X_{\eta_n(s)}^{(n)} \right|^{p \beta} \diffns s  \\
&+ d^{p/2} 5^{p-1} (1+\varepsilon)^p K^p (t-T_{j-1})^p \left( \frac{T}{n} \right)^{p \eta}.
\end{align*}
Since $\varepsilon >0$ is arbitrary, it may be chosen such that $c(p,d,\varepsilon):=d^{p/2-1} 5^{p-1} \varepsilon^p<1$.
Then we have
\begin{align}\label{esti:Lp:11}
\left|X_t-X_t^{(n)}\right|^p
&\leq
\frac{d^{p/2} 5^{p-1} (1+\varepsilon)^p}{(1-c(p,d,\varepsilon))} \left|X_{T_{j-1}}-X_{T_{j-1}}^{(n)}\right|^p \nonumber \\
&+ \frac{d^{p/2-1} 5^{p-1}}{(1-c(p,d,\varepsilon))} \sum_{i=1}^d \left| \int_{T_{j-1}}^t \left( \nabla u_{i,j}(s, X_s) - \nabla u_{i,j}(s,X_s^{(n)}) \right) \diffns W_s \right|^p \nonumber \\
&+ \frac{d^{p/2} 5^{p-1} ( 1+ \varepsilon )^p K^p(t-T_{j-1})^{p-1}}{(1-c(p,d,\varepsilon))}\int_{T_{j-1}}^t \left| X_{s}^{(n)}-X_{\eta_n(s)}^{(n)} \right|^{p\beta} \diffns s \nonumber \\
&+ \frac{d^{p/2} 5^{p-1} (1+\varepsilon)^p K^p (t-T_{j-1})^p}{(1-c(p,d,\varepsilon))} \left( \frac{T}{n} \right)^{p \eta}.
\end{align}
Taking the supremum and then expectation on both sides of \eqref{esti:Lp:11}, we have from Burkholder-Davis-Gundy's inequality and Jensen's inequality that
\begin{align*}%\label{esti:Lp:2}
\e\left[\sup_{T_{j-1} \leq u \leq t}\left|X_u-X_u^{(n)} \right|^p \right]
&\leq
\frac{d^{p/2} 5^{p-1} (1+\varepsilon)^p}{(1-c(p,d,\varepsilon))} \e\left[\left|X_{T_{j-1}}-X_{T_{j-1}}^{(n)}\right|^p \right] \nonumber \\
&+ \frac{d^{p/2-1} 5^{p-1} C(p,d) T^{\frac{p}{2}-1} }{(1-c(p,d,\varepsilon))} \sum_{i=1}^d \int_{T_{j-1}}^t \e\left[  \left| \nabla u_{i,j}(s, X_s) - \nabla u_{i,j}(s,X_s^{(n)}) \right|^p \right]\diffns s \nonumber \\
&+ \frac{d^{p/2} 5^{p-1} ( 1+ \varepsilon )^p K^p T^{p-1}}{(1-c(p,d,\varepsilon))} \int_{T_{j-1}}^t \e\left[ \left| X_{s}^{(n)}-X_{\eta_n(s)}^{(n)} \right|^{p\beta} \right]\diffns s \nonumber \\
&+ \frac{d^{p/2} 5^{p-1} (1+\varepsilon)^p (KT)^p}{(1-c(p,d,\varepsilon))} \left( \frac{T}{n} \right)^{p \eta}\nonumber,
\end{align*}
where $C(p,d)$ is the constant in Burkholder-Davis-Gundy's inequality.
From Lemma \ref{lem:stand} (i), we have
\begin{align*}
\e\left[\sup_{T_{j-1} \leq u \leq t}\left|X_u-X_u^{(n)} \right|^p \right]
&\leq
\frac{d^{p/2} 5^{p-1} \varepsilon^p}{(1-c(p,d,\varepsilon))} \e\left[\left|X_{T_{j-1}}-X_{T_{j-1}}^{(n)}\right|^p \right] \nonumber \\
&+ \frac{d^{p/2} 5^{p-1} C(p,d) T^{\frac{p}{2}-1}\varepsilon^p}{(1-c(p,d,\varepsilon))} \int_{T_{j-1}}
^t \e\left[\sup_{T_{j-1} \leq u \leq s }\left|X_u-X_u^{(n)} \right|^p\right]\diffns s \nonumber \\
&+ \frac{d^{p/2} 5^{p-1} ( 1+ \varepsilon )^p (KT)^p }{(1-c(p,d,\varepsilon))} \frac{C+T^{p \eta}}{n^{p \beta /2}}\\
&=C_1 \e\left[\left|X_{T_{j-1}}-X_{T_{j-1}}^{(n)}\right|^p \right]
+C_2 \int_{T_{j-1}}^t \e\left[ \sup_{T_{j-1} \leq u \leq s }\left|X_u-X_u^{(n)} \right|^p \right]\diffns s
+\frac{C_3}{n^{p \beta /2}}.
\end{align*}
Next, we prove by induction that for each $j=1,\ldots,m$,
\begin{align}\label{esti_T}
\e\left[\sup_{T_{j-1} \leq u \leq t}\left|X_u-X_u^{(n)} \right|^p \right]
&\leq
\frac{A_j}{n^{p \beta /2}},~t \in (T_{j-1},T_{j}],
\end{align}
where $A_1:=C_3e^{C_2T}$ and $A_j:=(C_1A_{j-1} +C_3)e^{C_2 T}$ for $j=2,\ldots, m$.
If $j=1,$ since $T_0=0$, for any $t \in (0,T_1]$, we have
\begin{align*}
\e\left[\sup_{0 \leq u \leq t}\left|X_u-X_u^{(n)} \right|^p \right]
&\leq
C_2 \int_{0}^t \e\left[\sup_{0 \leq u \leq s }\left|X_u-X_u^{(n)} \right|^p \right]\diffns s
+ \frac{C_3}{n^{p \beta /2}}.
\end{align*}
By Gronwall's inequality, we have
\begin{align*}
\e\left[\sup_{0 \leq u \leq t }\left|X_u-X_u^{(n)} \right|^p \right]
&\leq
\frac{C_3e^{C_2 T}}{n^{p \beta/2}}.
\end{align*}
Assume that \eqref{esti_T} holds for $j=1,\ldots,i-1$ with $2 \leq i \leq m$.
Then, for any $t \in (T_{i-1},T_i]$, we have
\begin{align*}
\e\left[\sup_{T_{i-1} \leq u \leq t}\left|X_u-X_u^{(n)} \right|^p\right]
&\leq
C_1 \e\left[\left|X_{T_{i-1}}-X_{T_{i-1}}^{(n)}\right|^p \right]
+C_2 \int_{T_{i-1}}^t \e\left[ \sup_{T_{i-1} \leq u \leq s}\left|X_u-X_u^{(n)} \right|^p \right]\diffns s
+\frac{C_3}{n^{p \beta /2}}\\
&\leq
C_2 \int_{T_{i-1}}^t \e\left[\sup_{T_{i-1} \leq u \leq s}\left|X_u-X_u^{(n)} \right|^p \right]\diffns s
+\frac{C_1A_{i-1}+C_3}{n^{p \beta /2}}.
\end{align*}
Using once more Gronwall's inequality, we have
\begin{align*}
\e\left[\sup_{T_{i-1} \leq u \leq t }\left|X_u-X_u^{(n)} \right|^p \right]
&\leq
\frac{(C_1A_{i-1} +C_3)e^{C_2 T}}{n^{p \beta/2}}
=\frac{A_i}{n^{p \beta/2}}.
\end{align*}
Hence \eqref{esti_T} holds for any $j=1,\ldots,m$ and  we have
\begin{align*}
	\e\left[\sup_{0 \leq s \leq T} \left|X_s-X_s^{(n)} \right|^p \right]
	&\leq \sum_{j=1}^{m}\e\left[\sup_{T_{j-1} \leq u \leq T_{j} }\left|X_u-X_u^{(n)} \right|^p \right]
	\leq \frac{1}{n^{p\beta}} \sum_{j=1}^{m} A_j.
\end{align*}
This concludes the statement of Theorem \ref{main_1.5}.
\qed

\subsection{Proof of Theorem \ref{main_2}}%Case of truncated symmetric $\alpha$-stable process}

%We first prove the rate of $L^p$-convergence for all $t \in [0,T]$.
%\begin{Lem}\label{main_trun_1.1}
%	Assume that we retain conditions of Theorem \ref{main_2}.
	%Assume $L$ is a $d$-dimensional truncated symmetric $\alpha$-stable process with $\alpha \in (1,2)$ and $d \geq 2$.
	%Under the same assumption of Theorem \ref{main_2},
%	Then for any $p \geq 1$, there exists a positive constant $C$ depends on $K,M,T,x_0,d,\alpha,\beta, \eta$ and $\|b\|_{C_b^{\beta}([0,T])}$ such that for any $t \in [0,T]$,
%	\begin{align}\label{lem:trun_stable:1}
%	\e\left[\left|X_t-X_t^{(n)}\right|^p\right]
%	\leq \left\{ \begin{array}{ll}
%	\displaystyle \frac{C}{n} &\text{ if } p \geq 2,~p \beta \geq 2,  \\
%	\displaystyle \frac{C}{n^{p\beta/2}} &\text{ if } p \geq 2,~1 \leq 	p \beta < 2 \text{ or } p \in [1,2).
%	\end{array}\right.
%	\end{align}
%\end{Lem}

Let us first remark that for $p \in [1,2)$, the $L^p$-norm is bounded by $L^2$-norm. Hence it is sufficient to prove the statement for $p \geq 2$.
As Theorem \ref{main_1.5}, let $\varepsilon \in (0,1)$ be given and consider the partition  $(T_j)_{j=0,\ldots,m}$ of closed interval $[0,T]$ defined in Lemma \ref{cor:stable:2}.
For $i=1,\ldots,d$ and $j=1,\ldots,m$, Corollary \ref{cor:stable:1} implies that there exists at least one solution $u_{i,j}$ to the backward Kolmogorov equation:
\begin{align*}%\label{Kolmo_trun_stable:}
	\frac{\partial u_{i,j}}{\partial t}  + \nabla u_{i,j} \cdot b + \mathcal{L} u_{i,j}=-b_i \text{ on } [T_{j-1},T_{j}] \times \real^d,~ u_{i,j}(T_{j},x)=0
\end{align*}
and $u_{i,j}$ satisfies,
\begin{align*}%\label{esti_nabla_u}
	\|\nabla u_{i,j}\|_{C_b^{\beta}([T_{j-1},T_j])}
	\leq C_0(T_j-T_{j-1})^{1-1/\alpha} \|b_j\|_{C_b^{\beta}([T_{j-1},T_j])}
	\leq C_0(T_j-T_{j-1})^{1-1/\alpha} \|b\|_{C_b^{\beta}([0,T])}
	\leq \varepsilon
\end{align*}
and
\begin{align}\label{esti_D2_u}
\|D^2u_{i,j}\|_{\infty} \leq M \|b\|_{C_b^{\beta}([0,T])}.
\end{align}
For any $t \in [T_{j-1},T_j]$, using It\^o's formula and Kolmogorov equation, we have
\begin{align}\label{rep:trun_stable}
\int_{T_{j-1}}^t b_i(s,X_s) \diffns s
&=u_{i,j}(T_{j-1},X_{T_j})-u_{i,j}(t,X_t) \notag \\
&+\int_{T_{j-1}}^t \int_{|z|\leq 1} \{ u_{i,j}(s,X_{s-} + z)-u_{i,j}(s,X_{s-}) \} \widetilde{N}(\diffns s,\diffns z)
\end{align}
and
\begin{align}\label{rep:trun_stable_EM}
\int_{T_{j-1}}^t b_i(s,X_s^{(n)}) \diffns s
&=u_{i,j}(T_{j-1},X_{T_j}^{(n)})-u_{i,j}(t,X_t^{(n)}) \notag\\
&+\int_{T_{j-1}}^t \int_{|z|\leq 1} \{ u_{i,j}(s,X_{s-}^{(n)} +z)-u_{i,j}(s,X_{s-}	^{(n)}) \} \widetilde{N}(\diffns s,\diffns z) \nonumber\\
&+\int_{T_{j-1}}^t \nabla u_{i,j} (s,X_s^{(n)}) \cdot (b(\eta_n(s), X_{\eta_n(s)}^{(n)}) - b(s, X_{s}^{(n)})) \diffns s.
\end{align}
%where $\gamma(z):=\1(|z| \leq 1) z$.
Combining\eqref{rep:trun_stable} and \eqref{rep:trun_stable_EM} and using similar arguments as in the case of Wiener process we have,
\begin{align}\label{esti:Lp:1_trun}
	\left|X_t-X_t^{(n)}\right|^p
	&\leq
	\frac{d^{p/2} 5^{p-1} (1+\varepsilon)^p}{(1-c(p,d,\varepsilon))} \left|X_{T_{j-1}}-X_{T_{j-1}}^{(n)}\right|^p \nonumber \\
	&+ \frac{d^{p/2-1} 5^{p-1}}{(1-c(p,d,\varepsilon))} \sum_{i=1}^d \left| \int_{T_{j-1}}^t \int_{|z|\leq 1} H_{i,j}(s,z) \widetilde{N}(\diffns s,\diffns z) \right|^p \nonumber \\
	&+ \frac{d^{p/2} 5^{p-1} ( 1+ \varepsilon )^p K^p(t-T_{j-1})^p}{(1-c(p,d,\varepsilon))}\int_{T_{j-1}}^t \left| X_{s}^{(n)}-X_{\eta_n(s)}^{(n)} \right|^{p\beta} \diffns s \nonumber \\
	&+ \frac{d^{p/2} 5^{p-1} (1+\varepsilon)^p K^p (t-T_{j-1})^p}{(1-c(p,d,\varepsilon))} \left( \frac{T}{n} \right)^{p \eta},
\end{align}
where
\begin{align*}
	H_{i,j}(s,z)
	:=\left\{ u_{i,j}(s,X_{s-} +z)-u_{i,j}(s,X_{s-}) \right\}
	-\left\{ u_{i,j}(s,X_{s-}^{(n)} +z)-u_{i,j}(s,X_{s-}^{(n)}) \right\}.
\end{align*}
Using twice the mean value theorem, we have
\begin{align*}
H_{i,j}(s,z)
&= \sum_{k=1}^d \int_0^1
(X_{s-}^k-X_{s-}^{(n,k)}) \\
&\times \left\{ \frac{\partial u_{i,j}}{\partial x_k} \left(s,X_{s-}^{(n)} + z+\tau(X_{s-}-X_{s-}^{(n)})\right) - \frac{\partial u_{i,j}}{\partial x_k}\left(s,X_{s-}^{(n)} + \tau(X_{s-}-X_{s-}^{(n)}) \right) \right\} \diffns \tau\\
&=\int_0^1 \diffns \theta \int_0^1 \diffns \tau \sum_{k,\ell=1}^d (X_{s-}^k-X_{s-}^{n,k}) z_{\ell} \frac{\partial^2 u_{i,j}}{\partial x_k \partial x_{\ell}} \left(X_{s-}^{(n)}+\tau (X_{s-}-X_{s-}^{(n)}) + \theta z \right).
\end{align*}
From Kunita's inequality (see \cite[Theorem 4.4.23]{Applebaum}) and \eqref{esti_D2_u}, we have
\begin{align}\label{esti_trun_st_int}
&\e\left[ \sup_{T_{j-1} \leq u\leq t} \left|\int_{T_{j-1}}^u \int_{|z|\leq 1} H_{i,j}(s,z) \widetilde{N}(\diffns s,\diffns z) \right|^p \right] \notag\\
&\leq C\e \left[ \left( \int_{T_{j-1}}^t \int_{|z|\leq 1} |H_{i,j}(s,z)|^2 \nu(\diffns z) \diffns s \right)^{p/2} \right]
+ C\e \left[ \int_{T_{j-1}}^t \int_{|z|\leq 1} |H_{i,j}(s,z)|^p \nu(\diffns z) \diffns s \right] \notag\\
&\leq C\e \left[ \left( \int_{T_{j-1}}^t \int_{|z|\leq 1} \left|X_s-X_s^{(n)}\right|^2 |z|^2 \nu(\diffns z) \diffns s \right)^{p/2} \right]
+ C\e \left[ \int_{T_{j-1}}^t \int_{|z|\leq 1} \left|X_s-X_s^{(n)}\right|^p |z|^p \nu(\diffns z) \diffns s \right] \notag\\
&\leq C \int_{T_{j-1}}^t \e\left[\left|X_s-X_s^{(n)}\right|^p\right]\diffns s.
\end{align}
Taking the supremum and then the expectation on both sides of \eqref{esti:Lp:1_trun} and using \eqref{esti_trun_st_int} and Lemma \ref{lem:stand} (ii), we have
\begin{align*}
&\e\left[\sup_{T_{j-1} \leq u \leq t}\left|X_u-X_u	^{(n)}\right|^p \right]\\
&\leq C_1 \e\left[\left|X_{T_{j-1}}-X_{T_{j-1}}^{(n)}\right|^p \right]
+C_2 \int_{T_{j-1}}^t \e\left[ \sup_{T_{j-1} \leq u \leq s} \left| X_u - X_u^{(n)} \right|^{p} \right]\diffns s\\
&+C \int_{T_{j-1}}^t \e\left[\left|X_s^{(n)}-X_{\eta_n(s)}^{(n)}\right|^{p \beta}\right]\diffns s
+ \frac{C}{n^{p\eta}}\\
&\leq C_1 \e\left[\left|X_{T_{j-1}}-X_{T_{j-1}}^{(n)}\right|^p \right]
+C_2 \int_{T_{j-1}}^t \e\left[ \sup_{T_{j-1} \leq u \leq s} \left| X_u - X_u^{(n)} \right|^{p} \right]\diffns s
+ \left\{ \begin{array}{ll}
\displaystyle \frac{C_3}{n} &\text{ if } p \beta \geq 2,  \\
\displaystyle \frac{C_3}{n^{p\beta/2}} &\text{ if } 1 \leq p \beta < 2.
\end{array}\right.
\end{align*}
Arguing as in the proof of \eqref{esti_T}, we have that for $t \in (T_{j-1},T_{j}]$,
\begin{align*}%\label{esti_T_trun}
\e\left[\sup_{T_{j-1} \leq u \leq t}\left|X_u-X_u	^{(n)}\right|^p\right]
&\leq
\left\{ \begin{array}{ll}
\displaystyle \frac{A_j}{n} &\text{ if } p \beta \geq 2,  \\
\displaystyle \frac{A_j}{n^{p\beta/2}} &\text{ if } 1 \leq p \beta < 2,
\end{array}\right.
\end{align*}
where $A_1:=C_3e^{C_2T}$ and $A_j:=(C_1A_{j-1} +C_3)e^{C_2 T}$ for $j=2,\ldots, m$.
This concludes the statement of Theorem \ref{main_2}.
\qed

\section*{Acknowledgements}

The Authors would like to thank an anonymous referee for his/her helpful comments.

This work was initiated when the first author visited Ritsumeikan University.
He will like to thank Prof Jir\^o Akahori for his hospitality and encouragement.
He also wants to thank the stochastic analysis group for providing nice working environment.

The second author is very grateful to Professor Arturo Kohatsu-Higa for his supports and encouragement.

The research leading to this result has received funding from the European Union's Seventh Framework Programme for research, technological development and demonstration under grant agreement no 318984-RARE

\end{document}